\documentclass[11pt,a4paper]{article}
\usepackage{amsmath}
\usepackage{amsfonts}
\usepackage{amssymb}
\usepackage{graphicx}
\usepackage{tabularx}
\usepackage{booktabs}
\usepackage{lscape}

\newtheorem{thm}{Theorem}[section]
\newtheorem{lemma}[thm]{Lemma}
\newtheorem{cor}[thm]{Corollary}

\newtheorem{de}{Definition}[section]

\newenvironment{proof} {\par \noindent \textbf{Proof: }}{\QED \par \bigskip \par}
\newcommand{\QED}{\hfill$\square$}
\newcommand{\rz}{\vspace{0.2cm}}

\begin{document}
\baselineskip 16pt

\phantom{.} \vskip 5cm

\begin{center}
{\LARGE \bf Eccentric Connectivity Index  \\[12pt]
            of Chemical Trees} \\
\bigskip
\bigskip

{\large \sc Aleksandar Ili\'c\ \footnotemark[3] }

\smallskip
{\em Faculty of Sciences and Mathematics, Vi\v{s}egradska 33, 18 000 Ni\v{s}, Serbia} \\
e-mail: {\tt aleksandari@gmail.com}

\bigskip
{\large \sc Ivan Gutman }

\smallskip
{\em Faculty of Science, University of Kragujevac, P.O. Box 60, 34000 Kragujevac, Serbia} \\
e-mail: {\tt gutman@kg.ac.rs}

\bigskip\medskip
{\small (Received May 25, 2009)}
\bigskip

\end{center}


\begin{abstract}
The eccentric connectivity index $\xi^c$ is a distance--based molecular
structure descriptor that was recently used for mathematical modeling of
biological activities of diverse nature. We prove that the broom has maximum
$\xi^c$ among trees with a fixed maximum vertex degree, and characterize such
trees with minimum $\xi^c$\,. In addition, we propose a simple linear
algorithm for calculating $\xi^c$ of trees.
\end{abstract}

\footnotetext[3] {Corresponding author.}

\baselineskip=0.30in

\section{Introduction}

Let $G$ be a simple connected graph with $n = |V|$ vertices.
For a vertex $v \in V (G)$\,, $deg (v)$ denotes the degree of $v$\,.
For vertices $v, u \in V$\,, the distance $d (v, u)$ is defined as
the length of a shortest path between $v$ and $u$ in $G$\,.
The eccentricity $\varepsilon (v)$ of a vertex $v$ is the maximum
distance from $v$ to any other vertex. \rz

Sharma, Goswami and Madan \cite{ShGoMa97} introduced a distance--based molecular
structure descriptor, which they named ``{\it eccentric connectivity index\/}''
and which they defined as
$$
\xi^c = \xi^c (G) = \sum_{v \in V (G)} deg (v) \cdot \varepsilon (v) \ .
$$

The index $\xi^c$ was successfully used for mathematical modeling of
biological activities of diverse nature \cite{DuGuMa08,GuSiMa02,KuSaMa04,SaMa00,SaMa03}.
Some mathematical properties of $\xi^c$ were recently reported in \cite{ZhDu09}.

Chemical trees (trees with maximum vertex degree at most four) provide the graph
representation of alkanes \cite{GuPo86}. It is therefore a natural problem to
study trees with bounded maximum degree.

Denote by $\Delta = \Delta(T)$ the maximum vertex degree of a tree $T$\,. The
path $P_n$ is the unique $n$-vertex tree with $\Delta = 2$\,, while the star $S_n$ is
the unique tree with $\Delta = n-1$\,. Therefore, we can assume that
$3 \leq \Delta \leq n - 2$\,.

For an arbitrary tree $T$ on $n$ vertices \cite{ZhDu09},
$$
\left \lfloor \frac{3 (n - 1)^2 + 1}{2} \right \rfloor = \xi^c (P_n)
\geq \xi^c (T) \geq \xi^c (S_n) = 3 (n - 1) \ .
$$

\vspace{5mm}

\section{Chemical trees with maximum eccentric connectivity index}

The broom $B_{n, \Delta}$ is a tree consisting of a star $S_{\Delta
+ 1}$ and a path of length $n - \Delta - 1$ attached to a
pendent vertex of the star. It is proven
in \cite{LiGu07} that among trees with maximum vertex degree equal to
$\Delta$\,, the broom $B_{n, \Delta}$ uniquely minimizes the largest eigenvalue of the
adjacency matrix. Further, within the same class of trees, the broom has minimum
Wiener index and Laplacian-energy like invariant \cite{St09}. In \cite{YaYe05} and
\cite{YuLv06} it was demonstrated that the broom has minimum energy among trees with,
respectively, fixed diameter and fixed number of pendent vertices.
\vspace{0.2cm}

\begin{figure}[ht]
  \center
  \includegraphics [width = 9cm]{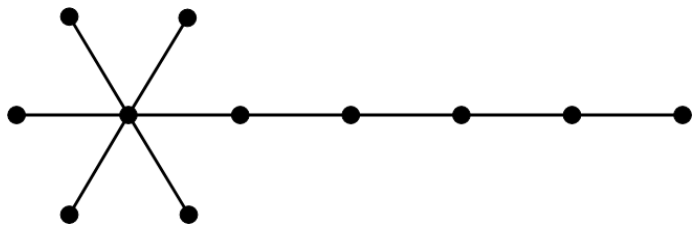}
  \caption { \textit{ The broom $B_{11, 6}$\,. } }
\end{figure}

The $\Delta$-starlike tree $T(n_1,n_2,\ldots,n_\Delta)$ is a
tree composed of the root $v$\,, and the paths $P_{n_1}$\,,
$P_{n_2}$\,, \ldots, $P_{n_\Delta}$\,, attached to $v$\,.
The number of vertices of $T(n_1,n_2,\ldots,n_{\Delta})$ is
thus equal to $n_1 + n_2 + \cdots + n_{\Delta} +
1$\,. Notice that the broom $B_{n, \Delta}$ is
a $\Delta$-starlike tree, $B_{n, \Delta} \cong T(n-\Delta,1,1,\ldots,1)$\,.

\begin{thm}
\label{thm-pi} Let $w$ be a vertex of a nontrivial connected graph
$G$\,. For nonnegative integers $p$ and $q$\,, let $G (p, q)$ denote
the graph obtained from $G$ by attaching to the vertex $w$ pendent paths
$P = w v_1 v_2 \ldots v_p$ and $Q = w u_1 u_2 \dots u_q$ of lengths $p$
and~$q$\,, respectively. If $p \geq q \geq 1$\,, then
$$
\label{eq-pi} \xi^c (G (p, q)) < \xi^c (G (p + 1, q - 1)) \ .
$$
\end{thm}

\begin{proof}
The degrees of vertices $u_{q - 1}$ and $v_p$ are changed, while all other vertices have the same
degree in $G (p + 1, q - 1)$ as in $G (p, q)$\,. Since after this transformation
the longer path has increased, the eccentricity of vertices from $G$ are either
the same or increased by one. We will consider three cases based on the longest path
from the vertex $w$ in the graph $G$\,. Denote by $deg' (v)$ and $\varepsilon'(v)$
the degree and eccentricity of vertex $v$ in $G(p+1,q-1)$\,.

\noindent
{\bf Case 1. } The length of the longest path from the vertex $w$ in $G$ is greater than $p$\,.
This means that the pendent vertex of $G$\,, most distant from $w$ is the most
distant vertex for all vertices of $P$ and $Q$\,. It follows that
$\varepsilon_{G (p + 1, q - 1)} (v) = \varepsilon_{G (p, q)} (v)$ for all vertices
$w, v_1, v_2, \ldots, v_p, u_1, u_2, \ldots, u_{q - 1}$\,, while the eccentricity
of $u_q$ increased by $p + 1 - q$\,.
\begin{eqnarray*}
\xi^c (G (p + 1, q - 1)) - \xi^c (G (p, q)) &\geq&
\left[ deg' (u_{q - 1})\,\varepsilon' (u_{q - 1})
+ deg' (u_{q})\,\varepsilon'\,(u_{q}) + deg' (v_{p})\,\varepsilon' (v_{p}) \right] \\
&-& \left[ deg (u_{q - 1})\,\varepsilon (u_{q - 1}) + deg (u_{q})\,\varepsilon (u_{q}) +
deg (v_{p})\,\varepsilon (v_{p}) \right] \\
&=& - \varepsilon (u_{q - 1}) + (p - q + 1) + \varepsilon (v_p) > 0 \ .
\end{eqnarray*}

\noindent {\bf Case 2. } The length of the longest path from the vertex $w$ in $G$ is less than or
equal to $p$ and greater than $q$\,. This means that either the vertex of $G$ that is most distant
from $w$ or the vertex $v_p$ is the most distant vertex for all vertices of $P$\,, while for
vertices $w, u_1, u_2, \ldots, u_q$ the most distant vertex is $v_p$\,. It follows that
$\varepsilon_{G (p + 1, q - 1)} (v) = \varepsilon_{G (p, q)} (v)$ for vertices $v_1, v_2, \ldots,
v_p$\,, while $\varepsilon_{G (p + 1, q - 1)} (v) = \varepsilon_{G (p, q)} (v) + 1$ for vertices
$w, u_1, u_2, \ldots, u_{q - 1}$\,. The eccentricity of $u_q$ increased by at least $1$\,.
\begin{eqnarray*}
\xi^c (G (p + 1, q - 1)) - \xi^c (G (p, q)) &\geq& deg' (w)\,\varepsilon' (w) +
deg' (v_{p})\,\varepsilon' (v_{p}) + \sum_{j = 1}^q deg' (u_{j})\,\varepsilon' (u_{j})\\
&-& deg (w)\,\varepsilon (w) - deg (v_{p}) \,\varepsilon (v_{p}) -
\sum_{j = 1}^q deg (u_{j})\,\varepsilon (u_{j})\\
&\geq& q + \left[ \varepsilon (u_{q - 1}) + 1 \right]
\left[ deg (u_{q - 1}) - 1 \right] -
\varepsilon (u_{q - 1})\,deg (u_{q - 1}) + \varepsilon (v_p)\\
&>& \varepsilon (v_p) - \varepsilon (u_{q - 1}) > 0 \ .
\end{eqnarray*}

\noindent
{\bf Case 3. } The length of the longest path from the vertex $w$ in $G$ is
less than or equal to $q$\,. This means that the pendent vertex most distant from
the vertices of $P$ and $Q$ is either $v_p$ or $u_q$\,, depending on the position.
Using the formula for eccentric connectivity index of a path, we have
\begin{eqnarray*}
\xi^c (G (p + 1, q - 1)) - \xi^c (G (p, q)) &>& \xi^c (P_{p + q + 1}) +
[deg (w) - 2]\,\varepsilon' (w) \\
&-& \xi^c (P_{p + q + 1}) - [deg (w) - 2]\,\varepsilon (w) \\
&=& deg (w) - 2 \geq 0 \ .
\end{eqnarray*}
Since $G$ is a nontrivial graph with at least one vertex, we have strict inequality.

This completes the proof.
\end{proof}

\begin{thm}
\label{thm-broom} Let $T \not \cong B_{n, \Delta}$ be an arbitrary
tree on $n$ vertices with maximum vertex degree $\Delta$\,. Then
$$
\xi^c (B_{n, \Delta}) > \xi^c (T) \ .
$$
\end{thm}

\begin{proof}
Fix a vertex $v$ of degree $\Delta$ as a root and let
$T_1, T_2, \ldots, T_{\Delta}$ be the trees
attached at~$v$\,. We can repeatedly apply the transformation described
in Theorem \ref{thm-pi} to any vertex of degree at least three with
greatest eccentricity from the root in every tree~$T_i$\,, as long as
$T_i$ does not become a path. When all trees $T_{1} ,T_{2},\dots, T_{\Delta}$
turn into paths, we can again apply transformation from Theorem~\ref{thm-pi}
at the vertex~$v$ as long as there exists at least two paths of length greater
than one, further decreasing the eccentric connectivity index. Finally, we arrive
at the broom $B_{n, \Delta}$ as the unique tree with maximum eccentric
connectivity index.
\end{proof}

By direct verification, it holds
$$
\xi^c (BT_{n, \Delta}) = \left \lfloor \frac{3n^2 - 2\Delta n - 2n -
\Delta^2 +4\Delta}{2} \right \rfloor .
$$

From the above proof, we also get that
$B'_{n,\Delta} = T(n-\Delta-1,2,1,\ldots,1)$ has the second minimal
$\xi^c$ among trees with maximum vertex degree $\Delta$\,.

It was proven in \cite{ZhDu09} that the path $P_n$ has maximum and the
star $S_n$ minimum $\xi^c$-value among connected graphs on $n$ vertices.
From Theorem~\ref{thm-broom} we know that the maximum eccentric connectivity
index among trees on $n$~vertices is achieved for one of the brooms
$B_{n,\Delta}$\,. If $\Delta>2$\,, we
can apply the transformation from Theorem~\ref{thm-pi} at the
vertex of degree~$\Delta$ in $B_{n, \Delta}$ and obtain
$B_{n, \Delta-1}$\,. Thus, it follows
$$
EE(S_{n}) = EE(B_{n,n-1}) < EE(B_{n,n-2}) < \cdots < EE(B_{n,3})<
EE(B_{n,2})=EE(P_{n}) \ .
$$

Also, it follows that $B_{n, 3}$ has the second maximum eccentric connectivity
index among trees on $n$ vertices.

\section{The minimum eccentric connectivity index of trees with fixed \\radius}

Vertices of minimum eccentricity form the center. A tree has exactly one
or two adjacent center vertices; in this latter case one speaks of a
bicenter. In what follows, if a tree has a bicenter, then our considerations
apply to any of its center vertices.

For a tree $T$ with radius $r(T)$\,,
$$
\label{tree_center} d (T) = \left\{
\begin{array}{l l}
  2\,r(T) - 1 & \quad \mbox{if $T$ has a bicenter }\\[3mm]
  2\,r (T) & \quad \mbox{if $T$ has has a center. }\\
\end{array} \right.
$$

Let $T_{(n, d)}$ be the set of $n$-vertex trees obtained from the
path $P_{d + 1} = v_0 v_1 \ldots v_d$ by attaching
$n - d - 1$ pendent vertices to $v_{\lfloor d/2 \rfloor}$ and/or
$v_{\lceil d/2 \rceil}$\,, where $2 \leq d \leq n - 2$\,.
Zhou and Du in \cite{ZhDu09} proved that for arbitrary tree $T$ on $n$
vertices and diameter $d$\,,
$$
\xi^c (T) \geq \xi (T^*)\ , \quad T^* \in  T_{(n, d)}
$$
with equality if and only if $T \in T_{(n, d)}$\,.
Using the transformation from Theorem \ref{thm-pi} and applying it to a
center vertex, it follows that $\xi^c (T') < \xi^c (T'')$ for $T' \in T_{(n, 2r-1)}$
and $T'' \in T_{(n, 2r)}$\,.

\begin{cor}
Let $T$ be an arbitrary tree on $n$ vertices with radius $r$\,. Then
$$
\xi^c (T) \geq 3r (2r - 1) + 2 + (n - 2r)(2r + 1)
$$
with equality if and only if $T \in T_{(n, 2r-1)}$\,.
\end{cor}

\vspace{5mm}

\section{The maximum eccentric connectivity index of trees with \\ perfect matchings}

A graph possessing perfect matchings must have an even number of vertices.
Therefore throughout this section we assume that $n$ is even.

It is well known that if a tree $T$ has a perfect matching,
then this perfect matching $M$ is unique:
namely, a pendent vertex $v$ has to be matched with its unique neighbor $w$\,,
and then $M-\{vw\}$ forms the perfect matching of $T-v-w$\,.

Let $A_{n, \Delta}$ be a $\Delta$-starlike tree $T(n-2\,\Delta,2,2,\ldots,2,1)$
consisting of a central vertex $v$\,,
a pendent vertex, a pendent path of length $n-2\,\Delta$\,, and
$\Delta - 2$ pendant paths of length $2$\,, all attached to $v$\,.

\begin{thm}
The tree $A_{n,\Delta}$ has maximum eccentric connectivity index among trees
with perfect matching and maximum vertex degree $\Delta$\,.
\end{thm}

\begin{proof}
Let $T$ be an arbitrary tree with perfect matching and let
$v$ be a vertex of degree $\Delta$\,,
with neighbors $v_1, v_2, \ldots, v_{\Delta}$\,.
Let $T_1, T_2, \ldots, T_{\Delta}$ be the maximal subtrees
rooted at $v_1, v_2, \ldots, v_{\Delta}$\,, respectively.
Then at most one of the numbers $|T_1|, |T_2|, \ldots, |T_{\Delta}|$
can be odd (if $T_i$ and $T_j$ have odd number of vertices, then their
roots $v_i$ and $v_j$ will be unmatched). Since the number of vertices
of $T$ is even, there exists exactly one among
$T_1, T_2, \ldots,T_{\Delta}$ with odd number of vertices.

Using Theorem \ref{thm-pi}, we may transform each $T_i$ into a path
attached to $v$ -- while simultaneously decreasing $\xi^c$ and keeping the
existence of a perfect matching.
Assume that $T_{\Delta}$ has odd number of vertices,
while the remaining trees have even number of vertices.
We apply a transformation similar to the one in Theorem \ref{thm-pi},
but instead of moving one vertex, we move two vertices
in order to keep the existence of a perfect matching.
Thus, if $p \geq q \geq 2$ then
$$
\xi^c (G (p, q)) < \xi^c (G (p + 2, q - 2)) \ .
$$
Using this transformation we may reduce $T_{\Delta}$ to one vertex,
the trees $T_2, \ldots, T_{\Delta - 1}$ to two vertices,
leaving $T_1$ with $n - 2\Delta$ vertices, and thus obtaining $A_{n,\Delta}$\,.
Since all times we strictly decreased $\xi^c$\,,
we conclude that $A_{n, \Delta}$ has minimum eccentric connectivity index
among the trees with perfect matching and maximum vertex degree $\Delta$\,.
\end{proof}

The path $P_n \cong A_{n,2}$ has maximum, while $A_{n,n/2}$ has minimum
eccentric connectivity index among trees with perfect matchings.

\vspace{5mm}

\section{The minimum eccentric connectivity index of trees  \\
with fixed number of pendent vertices}

In \cite{ZhDu09} the authors determinate the $n$-vertex trees with $p$ pendent vertices,
$2 \leq p \leq n - 1$\,, with the maximum eccentric connectivity index, and, consecutively, the extremal trees
with the maximum, second-maximum and third-maximum eccentric connectivity index for $n \geq 6$\,.
For the completeness, here we determine the $n$-vertex trees with $2 \leq p \leq n - 1$ pendent vertices
that have minimum eccentric connectivity index.

\newpage

\begin{de}
Let $v$ be a vertex of a tree $T$ of degree $m + 1$\,. Suppose that $P_1, P_2, \ldots, P_m$ are
pendent paths incident with $v$\,, with lengths $1 \leq n_1 \leq n_2 \leq \ldots \leq n_m$\,. Let
$w$ be the neighbor of $v$ distinct from the starting vertices of paths $v_1, v_2, \ldots, v_m$\,,
respectively. We form a tree $T' = \delta (T, v)$ by removing the edges $v v_1, v v_2, \ldots, v
v_{m - 1}$ from $T$ and adding $m - 1$ new edges $w v_1, w v_2, \ldots, w v_{m - 1}$ incident with
$w$\,. We say that $T'$ is a $\delta$-transform of $T$ and write $T' = \delta (T, v)$\,.
\end{de}

\begin{thm}
\label{thm-delta} Let $T' = \delta (T, v)$ be a $\delta$-transform of a tree $T$ of order $n$\,.
Let $v$ be a non-central vertex, furthest from the root among all branching vertices (with degree
greater than~$2$). Then
$$
\xi^c (T) > \xi^c (T')\,.
$$
\end{thm}

\begin{proof}
The degrees of vertices $v$ and $w$ have changed -- namely, $deg (v) - deg' (v) = deg' (w) - deg
(w) = m - 1$\,. Since the furthest vertex from $v$ does not belong to $P_1, P_2, \ldots, P_m$ and
$n_m \geq n_i$ for $i = 1, 2, \ldots, m - 1$\,, it follows that the eccentricities of all vertices
different from $P_1, P_2, \ldots, P_{m - 1}, P_m$ do not change after $\delta$ transformation. The
eccentricities of vertices from $P_m$ also remain the same, while the eccentricities of vertices
from $P_1, P_2, \ldots, P_{m - 1}$ decrease by one. Using the equality $\varepsilon (v) =
\varepsilon (w) + 1$\,, it follows that
\begin{eqnarray*}
\xi^c (T) - \xi^c (T') &=& \sum_{i = 1}^{m - 1} (1 + 2(n_i - 1)) + (m - 1) \cdot \varepsilon (v) - (m - 1) \cdot \varepsilon (w) \\
&=& 2 \left ( n_1 + n_2 + \ldots + n_{m - 1} \right) - (m - 1) + (m - 1)(\varepsilon (v) - \varepsilon (w)) \\
&=& 2 \left ( n_1 + n_2 + \ldots + n_{m - 1} \right) > 0\,.
\end{eqnarray*}
This completes the proof.
\end{proof}

The $p$-starlike tree $SB_{n, p} = T(n_1, n_2, \ldots, n_p)$ is {\it balanced\/} if all paths have
almost equal lengths, i.e., $|n_i - n_j| \leq 1$ for every $1 \leq i \leq j \leq p$\,.

\begin{thm}
The balanced $p$-starlike tree $SB_{n, p}$ has minimum eccentric connectivity index among trees
with $p$ pendent vertices, $2 < p < n - 1$\,.
\end{thm}

\begin{proof}
Let $T$ be a rooted $n$-vertex tree with $p$ pendent vertices. If
$T$ contains only one vertex of degree greater than two, we can
apply Theorem \ref{thm-pi} in order to arrive at the balanced starlike
tree $SB_{n, p}$\,, without changing the
number of pendent vertices. If $T$ has several vertices of degree
greater than $2$\,, such that there are only pendent paths attached
below them, then we take the one most distant from the center vertex of $T$\,.
By repetitive application of the $\delta$ transformation and
balancing pendant paths, the eccentric connectivity index decreases.

Assume that we arrived at a tree with two centers $C = \{v, w\}$ with
only pendent paths attached at both centers.
If all pendent paths have equal lengths, then $n = k p + 2$\,. Since we
can reattach $p - 2$ pendent paths at any central vertex
without changing $\xi^c (T)$\,, it follows that there are exactly
$\lfloor p/2 \rfloor$ extremal trees with minimum
eccentric connectivity index in this special case.

Now, let $R$ be the path with length $r = r (T) - 1$ attached to
$v$ and let $Q$ be the shortest path of length $q$ attached to
$w$\,. After applying the $\delta$ transformation at vertex $v$\,, the
eccentric connectivity index remains the same. If we apply the transformation
from Theorem \ref{thm-pi} to two pendant paths of lengths $r + 1$ and $q$ attached at $w$\,,
we will strictly decrease the eccentric connectivity index. Finally, we conclude
that $SB_{n, p}$ is the unique extremal tree that minimizes $\xi^c$ among
$n$-vertex trees with $p$ pendent vertices for $n \not \equiv 2 \pmod p$\,.
\end{proof}

\section{Chemical trees with minimal eccentric connectivity index}

\begin{thm}
\label{thm-rot}
Let $T$ be a rooted tree, with a center vertex $c$ as root. Let $u$ be the vertex
closest to the root vertex, such that $deg (u) < \Delta$\,. Let $w$ be the pendent
vertex most distant from the root, adjacent to vertex $v$\,, such that
$\varepsilon (v) > \varepsilon (u)$\,. Construct a tree $T'$ by deleting the edge $vw$ and
inserting the new edge $uw$\,. Then
$$
\xi^c (T) > \xi^c (T') \ .
$$
\end{thm}

\begin{proof}
In the transformation $T \to T'$ the degrees of vertices other than $u$
and $v$ remain the same, while $deg' (u) = deg (u) + 1$ and $deg' (v) = deg (v) - 1$\,.
Since the tree is rooted at the center vertex, the radius of $T$ is equal
to $r (T) = d (c, w)$\,. Furthermore, there exists a vertex $w'$ in a different
subtree attached to the center vertex, such that
$d (c, w') = r (T)$ or $d (c, w') = r (T) - 1$\,.
From the condition $\varepsilon (v) > \varepsilon (u)$\,, it follows that
$d (c, w') > d (c, u)$ and $w' \neq u$\,.

By rotating the edge $vw$ to $uw$\,, the eccentricity of vertices other than $w$ decrease if and
only if $w$ is the only vertex at distance $r (T)$ from the center vertex. Otherwise, the
eccentricities remain the same. In both cases, we have
\begin{eqnarray*}
\xi^c (T) - \xi^c (T') &\geq& deg (v)\,\varepsilon (v) + deg (w)\,\varepsilon (w)
+ deg (u)\,\varepsilon (u)\\
&-& \left[ deg' (v)\,\varepsilon' (v) + deg' (w)\,\varepsilon' (w) +
deg' (u)\,\varepsilon' (u) \right] \\
&\geq& \varepsilon (v) + (\varepsilon (v) - \varepsilon (u)) - \varepsilon (u)
= 2 (\varepsilon (v) - \varepsilon (u)) > 0 \ .
\end{eqnarray*}
This completes the proof.
\end{proof}

The Volkmann tree $VT (n, \Delta)$ is a tree on $n$ vertices and
maximum vertex degree $\Delta$\,, defined as follows \cite{770,FiHo02}.
Start with the root having $\Delta$ children. Every vertex different
from the root, which is not in one of the last two levels, has exactly
$\Delta -1$ children. In the last level, while not all vertices need
to exist, the vertices that do exist fill the level consecutively.
Thus, at most one vertex on the level second to last has its degree
different from $\Delta$ and $1$\,. For more details on Volkmann trees
see \cite{770,FiHo02,GuFMG07}. In \cite{770,FiHo02} it was shown that
among trees with fixed $n$ and $\Delta$\,, the Volkmann tree has
minimum Wiener index. Volkmann trees have also other extremal
properties among trees with fixed $n$ and $\Delta$
\cite{GuFMG07,790,SiTo05,YuLu08}.

\begin{figure}[ht]
  \center
  \includegraphics [width = 6cm]{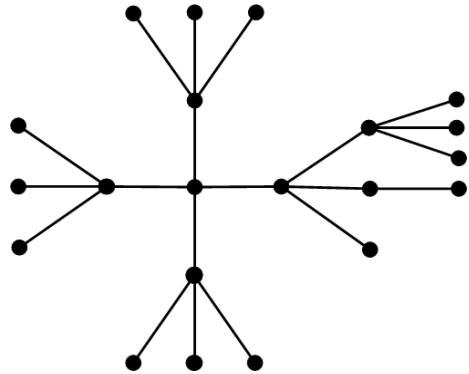}
  \caption { \textit{ The Volkmann tree $VT (21, 4)$\,. } }
\end{figure}

\begin{thm}
\label{thm-volkman} Let $T$ be an arbitrary
tree on $n$ vertices with maximum vertex degree~$\Delta$\,. Then
$$
\xi^c (T) \geq \xi^c (VT_{n, \Delta}).
$$
\end{thm}

\begin{proof}
Among $n$-vertex trees with maximum degree $\Delta$\,, let $T^*$ be the extremal tree with minimum
eccentric connectivity index. Assume that $u$ is a vertex closest to the root vertex $c$\,, with
$deg (u) < \Delta$ and let $w$ be the pendent vertex most distant from the root, adjacent to vertex
$v$\,. Also, let $k$ be the greatest integer, such that
$$
n \geq 1 + \Delta + \Delta (\Delta - 1) + \Delta (\Delta - 1)^2 +
\cdots + \Delta (\Delta - 1)^{k - 1} \ .
$$

First, we will show that the radius of $T^*$ has to be less than or equal to $k + 1$\,.
Assume that $r (T^*) = d (c, w) > k + 1$\,. Since the distance from the
center vertex to $u$ is less than or equal to $k$\,, it follows that
$$
\varepsilon (v) \geq 2 r (T^*) - 2 \geq k + r (T^*) \geq \varepsilon (u) \ .
$$
If strict inequality holds, then we can apply Theorem \ref{thm-rot} and decrease the eccentric
connectivity index -- which contradicts to the assumption that $T^*$ is the tree with minimum
$\xi^c$\,. Therefore, $\varepsilon (v) = \varepsilon (u)$ and after performing the transformation
from Theorem \ref{thm-rot}, the eccentric connectivity index does not change. According to the
definition of the number $k$\,, after finitely many transformations, the vertex~$w$ will be the
only vertex at distance $r(T)$ from the center vertex and we will strictly decrease $\xi^c
(T^*)$\,. Also, this means that for the case $n = 1 + \Delta + \Delta (\Delta - 1) + \Delta (\Delta
- 1)^2 + \cdots + \Delta (\Delta - 1)^{k - 1}$\,, the Volkmann tree is the unique tree with minimum
eccentric connectivity index.

Now, we can assume that the radius of $T^*$ is equal $k + 1$\,.
If the distance $d (c, u)$ is less than $k - 1$\,, it follows again
that $\varepsilon (v) > \varepsilon (u)$\,, which is
impossible. Therefore, the levels $1, 2, \ldots, k - 1$ are full
(level $i$ contains exactly $\Delta (\Delta - 1)^{i - 1}$ vertices),
while the $k$-th and $(k + 1)$-th levels contain
$$
L = n - \left[ 1 + \Delta + \Delta (\Delta - 1) + \Delta (\Delta - 1)^2 +
\cdots + \Delta (\Delta - 1)^{k - 1} \right]
$$
vertices.

Assume that $T^*$ has only one center vertex -- then $d (c, w) = k + 1$ and
$\varepsilon (v) = 2 r (T^*) - 1$\,.
If $d (c, u) = k - 1$\,, we can apply the transformation from Theorem \ref{thm-rot}
and strictly decrease $\xi^c$\,. Thus,
for $L > (\Delta - 1)^k$\,, the $k$-th level is also full and the pendent vertices in
the $(k + 1)$-th level can be
arbitrarily assigned. Using the same argument, for $L \leq (\Delta - 1)^k$\,, the extremal
trees are bicentral. By completing the $k$-th level,
we do not change the eccentric connectivity index -- since $\varepsilon (v) = \varepsilon (u)$\,.
Finally, $\xi (T^*) = \xi (VT (n, \Delta))$ and the result follows.
\end{proof}

In Table 1 we give the minimum value of eccentric connectivity index
among $n$ vertex trees with maximum vertex degree $\Delta$\,, together
with the number of such extremal trees (of which one is the Volkman tree).
Note that for $n \leq 2 \Delta$ the number of extremal trees is $1$\,, and for
$\Delta > 2$ holds $\xi^c (VT (n, \Delta - 1)) \geq \xi^c (VT (n, \Delta))$\,.

\vspace{5mm}

\section{A linear algorithm for calculating the eccentric connectivity index of a tree}

Let $T$ be a rooted tree, with a center vertex as root.
Let $c_1, c_2, \ldots, c_k$ be the
neighbors of the center vertex $c$\,, and $T_1, T_2, \ldots, T_k$ be the
corresponding rooted subtrees. Let $r_i$ be the length of the longest path
from $c_i$ in the subtree $T_i$\,, $i = 1, 2, \ldots, k$\,.

\begin{lemma}
\label{le-ecc}
The eccentricity of the vertex $v \in V (T_i)$ equals
$$
\varepsilon (v) = d (v, c) + 1 + \max_{i \neq k} r_k \ .
$$
\end{lemma}

\vspace{5mm}

\begin{tabular} {||c|ccccccccc||} \hline
$n$ & $\Delta=2$ & $\Delta=3$ & $\Delta=4$ & $\Delta=5$ & $\Delta=6$ & $\Delta=7$ & $\Delta=8$
& $\Delta=9$ & $\Delta=10$ \\[2mm] \hline
11 & 150\,;\,1 & 79\,;\,3  & 62\,;\,5  & 60\,;\,6  & 49\,;\,1  & 49\,;\,1  & 49\,;\,1  & 49\,;\,1  & 30\,;\,1 \\[1mm]
12 & 182\,;\,1 & 88\,;\,3  & 69\,;\,4  & 67\,;\,8  & 54\,;\,1  & 54\,;\,1  & 54\,;\,1  & 54\,;\,1  & 54\,;\,1 \\[1mm]
13 & 216\,;\,1 & 97\,;\,1  & 76\,;\,4  & 74\,;\,9  & 72\,;\,10  &59\,;\,1  & 59\,;\,1  & 59\,;\,1  & 59\,;\,1 \\[1mm]
14 & 254\,;\,1 & 106\,;\,1  & 83\,;\,3  & 81\,;\,11  & 79\,;\,12  & 64\,;\,1  & 64\,;\,1  & 64\,;\,1  & 64\,;\,1 \\[1mm]
15 & 294\,;\,1 & 130\,;\,7  & 90\,;\,2  & 88\,;\,11  & 86\,;\,16  & 84\,;\,14  & 69\,;\,1  & 69\,;\,1  & 69\,;\,1 \\[1mm]
16 & 338\,;\,1 & 141\,;\,10  & 97\,;\,1  & 95\,;\,12  & 93\,;\,19  & 91\,;\,19  & 74\,;\,1  & 74\,;\,1  & 74\,;\,1 \\[1mm]
17 & 384\,;\,1 & 152\,;\,7  & 104\,;\,1  & 102\,;\,11  & 100\,;\,23  & 98\,;\,24 & 96\,;\,21 & 79\,;\,1 & 79\,;\,1 \\[1mm]
18 & 434\,;\,1 & 163\,;\,7 & 138\,;\,24  & 109\,;\,11  & 107\,;\,25 & 105\,;\,31 & 103\,;\,27 & 84\,;\,1 & 84\,;\,1 \\[1mm]
19 & 486\,;\,1 & 174\,;\,4 & 147\,;\,20  & 116\,;\,9 & 114\,;\,29 & 112\,;\,37 & 110\,;\,36 & 108\,;\,29 & 89\,;\,1 \\[1mm]
20 & 542\,;\,1 & 185\,;\,3 & 156\,;\,18  & 123\,;\,8 & 121\,;\,30 & 119\,;\,46 & 117\,;\,45 & 115\,;\,39 & 94\,;\,1 \\[1mm]
&&&&&&&&& \\ \hline \hline
$n$ & $\Delta=11$ & $\Delta=12$ & $\Delta=13$ & $\Delta=14$ & $\Delta=15$ & $\Delta=16$ &
$\Delta=17$ & $\Delta=18$ & $\Delta=19$ \\[2mm] \hline
11 &         &         &         &         &         &         &         &         &        \\[1mm]
12 & 33\,;\,1 &         &         &         &         &         &         &         &        \\[1mm]
13 & 59\,;\,1  & 36\,;\,1  &         &         &          &        &         &         &     \\[1mm]
14 & 64\,;\,1  & 64\,;\,1  & 39\,;\,1  &          &          &         &         &         & \\[1mm]
15 & 69\,;\,1  & 69\,;\,1  & 69\,;\,1  & 42\,;\,1  &          &         &         &         &        \\[1mm]
16 & 74\,;\,1  & 74\,;\,1  & 74\,;\,1  & 74\,;\,1  & 45\,;\,1  &          &         &         &        \\[1mm]
17 & 79\,;\,1  & 79\,;\,1  & 79\,;\,1  & 79\,;\,1  & 79\,;\,1  & 48\,;\,1 &         &        &        \\[1mm]
18 & 84\,;\,1  & 84\,;\,1  & 84\,;\,1  & 84\,;\,1  & 84\,;\,1 & 84\,;\,1  & 51\,;\,1 &        &        \\[1mm]
19 & 89\,;\,1  & 89\,;\,1  & 89\,;\,1  & 89\,;\,1  & 89\,;\,1 & 89\,;\,1  & 89\,;\,1  & 54\,;\,1 &        \\[1mm]
20 & 94\,;\,1  & 94\,;\,1  & 94\,;\,1  & 94\,;\,1  & 94\,;\,1 & 94\,;\,1  & 94\,;\,1  & 94\,;\,1  & 57\,;\,1 \\[1mm]
\hline \hline
\end{tabular}

\vspace{5mm}

\baselineskip=0.20in

\noindent
{\bf Table 1.} The minimal value of the eccentricity connectivity index of trees
with $n$ vertices and maximum vertex degree $\Delta$\,, and the number of such extremal trees.

\vspace{15mm}

\baselineskip=0.30in

\begin{proof}
We show that the longest path starting at vertex $v$ has to traverse the center vertex $c$\,.
This means that the eccentricity of $v$ is equal to the sum of $d (v, c)$ and the
longest path starting at $c$ and not contained in $T_i$\,. Assume that the longest path
$P$ from $v$ stays in the subtree $T_i$\,, and let $w$ be the
vertex from $P$ at the smallest distance from the root $c$\,. Then
$d (v, c) \geq d (v, w) + 1$\,. Since the root vertex is a center of $T$\,, we have
$\max\limits_{k \neq i} r_k + 1 \geq r_i$ and consequently
$$
d (v, c) + \max_{k \neq i} r_k \geq d (v, w) + r_i \geq |P| \ .
$$
This means that $d (v, c) + 1 + \max\limits_{k \neq i} r_k$ is strictly greater
than $|P|$\,, which is a contradiction.
\end{proof}

We now present a simple linear algorithm for calculating the eccentric connectivity index of a
tree~$T$\,. First, find a center vertex of a tree -- this can be done in time $O (n)$ (see
\cite{CoLRS01} for details). For every vertex $v$\,, we have to find the length of the longest path
from $v$ in the subtree rooted at $v$\,. This can be done inductively using depth--first search,
also in time $O (n)$\,. If $r [v]$ represents the length of the longest path in the subtree rooted
at $v$\,, then
$$
r [v] = 1 + \max_{(v, w) \in E (T), \ w \neq p [v]} r [w]
$$
where $p [v]$ denotes the parent of vertex $v$ in $T$\,.
For all neighbors $c_i$ of the center vertex $c$\,, we can calculate the maximum
$\max\limits_{i \neq j} r [c_j]$\,. Finally, for every vertex $v$ we calculate the
eccentricity $\varepsilon (v)$ in $O (1)$ using Lemma \ref{le-ecc},
and sum $deg (v) \cdot \varepsilon (v)$\,.

The time complexity of the algorithm is linear $O (n)$\,, and the memory used
is $O (n)$\,, since we need three additional arrays of length $n$\,.

\vspace{5mm}

\noindent
{\it Acknowledgement.\/} This work was supported by the
research grants 144015G and 144007 of the Serbian Ministry of Science
and Technological Development.

\vspace{5mm}

\end{document}